\documentclass[11pt]{article}
\usepackage{amsthm, amsmath, amssymb, amsfonts, url, booktabs, tikz, setspace, fancyhdr, bm}
\usepackage{hyperref}
\usepackage{amsthm}
\usepackage{cancel}
\usepackage{geometry}
\usepackage{hyperref, enumerate}
\usepackage[shortlabels]{enumitem}
\usepackage[babel]{microtype}
\usepackage[english]{babel}
\usepackage[capitalise]{cleveref}
\usepackage{comment}
\usepackage{bbm}
\usepackage{csquotes}
\usepackage{mathabx}
\usepackage{tikz}
\usepackage{subcaption}
\usepackage{graphicx}
\usepackage{float}
\usepackage{xcolor}
\usetikzlibrary{positioning, arrows.meta, shapes.geometric}
\counterwithin{figure}{section}

\newtheorem{theorem}{Theorem}[section]

\newtheorem{conj}{Conjecture}
\newtheorem{claim}[theorem]{Claim}

\theoremstyle{definition}

\newtheorem*{defn-non}{Definition}

\definecolor{rosepink}{RGB}{255,102,204}
\definecolor{dateplum}{HTML}{993366}
\definecolor{darkdateplum}{RGB}{128,0,32}
\definecolor{lightdateplum}{RGB}{219,112,147}
\definecolor{darkred}{RGB}{139,0,0}
\definecolor{lightred}{RGB}{240,130,100}

\newtheorem{rmk}[theorem]{Remark}

\newlist{Case}{enumerate}{2}
\setlist[Case, 1]{%
    label           =   {\bfseries Case \arabic*.},
    labelindent=1em ,labelwidth=1.3cm, labelsep*=1em, leftmargin =!
}
\setlist[Case, 2]{%
    label           =   {\bfseries Subcase \arabic{Casei}.\arabic*.},
    labelindent=-1em ,labelwidth=1.3cm, labelsep*=1em, leftmargin =!
}
\newcommand{\Yemph}[1]{\textcolor{black}{\emph{#1}}}

\newenvironment{poc}{\begin{proof}[Proof of claim]}{\end{proof}}

\usepackage{todonotes}

\title{Bollob\'{a}s-Erd\H{o}s-Tuza conjecture for graphs with no induced $K_{s,t}$}

\author{Xinbu Cheng\thanks{Laboratory of Mathematics and Complex Systems, Ministry of Education, School of Mathematical Sciences, Beijing Normal University, Beijing, China. Emails: chengxinbu2006@sina.com.}
\and
Zixiang Xu\thanks{Extremal Combinatorics and Probability Group (ECOPRO), Institute for Basic Science (IBS), Daejeon, South Korea. Emails: zixiangxu@ibs.re.kr. Supported by IBS-R029-C4.}
}

\begin{document}
\date{}
\maketitle
\begin{abstract}
A widely open conjecture proposed by Bollob\'{a}s, Erd\H{o}s, and Tuza in the early 1990s states that for any $n$-vertex graph $G$, if the independence number $\alpha(G) = \Omega(n)$, then there is a subset $T \subseteq V(G)$ with $|T| = o(n)$ such that $T$ intersects all maximum independent sets of $G$. In this paper, we prove that this conjecture holds for graphs that do not contain an induced $K_{s,t}$ for fixed $t \ge s$. Our proof leverages the probabilistic method at an appropriate juncture.

\end{abstract}

\section{Introduction}
For a graph $G=(V,E)$, an \Yemph{independent set} is a set of vertices in $G$ such that no pair of vertices is adjacent. An independent set is a \emph{maximum independent set} if it has the largest possible size in the graph $G$, and this size is called the \Yemph{independence number} of $G$. Usually we write $\alpha(G)$ as the independence number of $G$. A subset $T\subseteq V(G)$ is a \Yemph{hitting set} of the graph $G$ if $T\cap I\neq\emptyset$ holds for each maximum independent set $I\subseteq V(G)$. Define $h(G)$ to be the smallest size of a hitting set $T\subseteq V(G)$. In this paper, we always assume that $n$ is sufficiently large whenever necessary. Bollob\'{a}s, Erd\H{o}s and Tuza~(see,~\cite{1999BookErdos,1991ErdosCollection}) proposed the following conjecture.
\begin{conj}[\cite{1999BookErdos,1991ErdosCollection}]\label{conj:BETConj}
    Let $G$ be an $n$-vertex graph with $\alpha(G)=\Omega(n)$, then $h(G)=o(n)$.
\end{conj}
According to Hajnal's observation~\cite{1965Hajnal}, this conjecture is trivial when $\alpha(G)>\frac{n}{2}$, but becomes challenging in general, even when $\alpha(G)=(\frac{1}{2}-\varepsilon)n$ for any $\varepsilon>0$. A feasible task is to investigate whether this conjecture holds for certain special classes of graphs. For example, Alon~\cite{2021AlonHittingSet,2023SIAM} demonstrated that every $n$-vertex regular graph $G$ with $\alpha(G)>(\frac{1}{4}+\varepsilon)n$ for any $\varepsilon>0$ satisfies $h(G)\le O_{\varepsilon}(\sqrt{n\log{n}})$. Recently, Cheng and Xu~\cite{2024ChengXu1} proved that this conjecture holds for several classes of graphs, including even-hole-free graphs, disk graphs, and circle graphs. We also recommend interested readers to explore some related results and approaches~\cite{2024HitInducedMathcing,2023P5Free}.

It has been proven in~\cite{2024ChengXu1} that for any $n$-vertex even-hole-free graph $G$ with linear-sized independence number, $h(G)\le \frac{5n}{\log{n}}$. Note that an even-hole-free graph does not contain any induced even cycles of length at least four. It is natural to ask whether we can strengthen the above result by only forbidding a single induced even cycle, such as an induced $4$-cycle. We answer this question for induced $C_{4}$-free graphs, albeit with a slightly weaker upper bound on $h(G)$ compared to the case of even-hole-free graphs. More generally, our main result shows that~\cref{conj:BETConj} holds for graphs without induced complete bipartite graph $K_{s,t}$ for any fixed $t\ge s$. 

\begin{theorem}\label{thm:Kst}
   For fixed positive integers $t\ge s$, let $G$ be an $n$-vertex induced $K_{s,t}$-free graph with $\alpha(G)=\Omega(n)$, then $h(G)=o(n)$.
\end{theorem}

\section{Proof of~\cref{thm:Kst}}
Let $G$ be an $n$-vertex induced $K_{s,t}$-free graph with $\alpha(G)=cn$ for some $0<c\le\frac{1}{2}$. Let $\delta>0$ be an arbitrary small number. It suffices to show that $h(G)<\delta n$ holds for any sufficiently large $n$. For a subset $W\subseteq V(G)$ and a vertex $v\in V(G)$, we use $N_{W}(v)$ to denote the neighborhood of $v$ in $W$ and sometimes we also use $N(v)$ if the subscript is clear. For subsets $A,B\subseteq V(G)$, we define $E(A,B)$ as the set of edges with one endpoint in $A$ and the other in $B$.

Since for every vertex $v\in V(G)$, $\{v\}\cup N_{G}(v)$ is a hitting set, we can assume that the minimum degree of $G$ is at least $\delta n-1$. We then arbitrarily pick a maximum independent set $I$. Firstly, we focus on the set of vertices in $V(G)\setminus I$ whose number of neighbors falls within a specific range, more precisely, for $1\le j\le\frac{2}{\delta}$, let
\begin{equation*}
    S_{j}:=\{v\in V(G)\setminus I : n(\log n)^{-(10s)^ {2j+1}}\le|N_{I}(v)|<n(\log n)^{-(10s)^ {2j-1}} \}.
\end{equation*}
Since $|V(G)\setminus I|=(1-c)n$, by pigeonhole principle, there exists some $1\le j\le\frac{2}{\delta}$ such that $|S_{j}|\le \frac{(1-c)\delta n}{2}<\frac{\delta n}{2}$. Fix this $j$.

We pick $k_{j}:=(\log{n})^{(10s)^{2j}}$ vertices in $I$ uniformly at random and let $I_{j}=\{u_{1},u_{2},\ldots,u_{k_{j}}\}\subseteq I$ be the subset consisting of those chosen vertices. For each $s$-tuple $(p_{1},\ldots,p_{s})\in\binom{[k_{j}]}{s}$, let $K^{(p_{1},\ldots,p_{s})}$ be the subgraph induced by the common neighborhood of $u_{p_{1}},\ldots,u_{p_{s}}$. Note that $G$ is induced $K_{s,t}$-free, therefore, for each $s$-tuple $(p_{1},\ldots,p_{s})\in\binom{[k_{j}]}{s}$, we have $\alpha(K^{(p_{1},\ldots,p_{s})})\le t-1$. Here we remark that the vertices in distinct $K^{(p_{1},\ldots,p_{s})}$ might be overlap, even identical. Let $K\subseteq V(G)$ be a subset consisting all vertices in $K^{(p_{1},\ldots,p_{s})}$ for all $s$-tuples $(p_{1},\ldots,p_{s})\in\binom{[k_{j}]}{s}$.

We then consider the edges between $I$ and $V(G)\setminus (I\cup K\cup S_{j})$, that is, the edge set $E(I,V(G)\setminus (I\cup K\cup S_{j}))$. Notice that this edge set depends on $K$, which in turn depends on the random choice of $I_{j}$. For convenience, let $e:=| E(I,V(G)\setminus (I\cup K\cup S_{j}))|$ be a random variable.  

Note that for a vertex $v\in I$, the average number of $|N_{V(G)\setminus(I\cup K\cup S_{j})}(v)|$ is $\frac{e}{cn}$. Therefore, there exists a subset $H\subseteq I$ of size $|H|= (t-1)\binom{k_{j}}{s}+1$ such that $|E(H,V(G)\setminus(I\cup K\cup S_{j})|\le \frac{\big((t-1)\binom{k_{j}}{s}+1\big)e}{cn}$. Fix this subset $H$ and let 
\begin{equation*}
    N(H):=\bigcup\limits_{v\in H}N_{V(G)\setminus(I\cup K\cup S_{j})}(v).
\end{equation*}

\begin{claim}\label{claim:Hitting}
    $H\cup N(H)\cup S_{j}$ hits all maximum independent sets of $G$.
\end{claim}
\begin{poc}
     Suppose there is a maximum independent set $I'\subseteq V(G)$ such that $I'\cap (H\cup N(H)\cup S_{j})=\emptyset$. Note that $|I'\setminus K|\ge cn-(t-1)\binom{k_{j}}{s}$ since a maximum independent set can intersect each $K^{(p_{1},\ldots,p_{s})}$ with at most $t-1$ vertices. Also note that $I'\cap N(H)=\emptyset$, therefore $H\cup (I'\setminus K)$ is an independent set of size $cn-(t-1)\binom{k_{j}}{s}+(t-1)\binom{k_{j}}{s}+1>cn$, a contradiction to that $I'$ is a maximum independent set.
\end{poc}
By choices of $H$ and $S_{j}$, we can see that $|H\cup N(H)\cup S_{j}|<(t-1)\binom{k_{j}}{s}+1+ 
\frac{\big((t-1)\binom{k_{j}}{s}+1\big)e}{cn} +\frac{\delta n}{2}$. By~\cref{claim:Hitting}, if $(t-1)\binom{k_{j}}{s}+1+ 
\frac{\big((t-1)\binom{k_{j}}{s}+1\big)e}{cn} +\frac{\delta n}{2}<\delta n$, that is, if $e<\frac{\delta cn^{2}}{2(t-1)\binom{k_{j}}{s}+2}-cn$, then we are done.

It suffices to provide a desired upper bound on $\mathbb{E}[e]$, which is the expected number of edges in $E(I,V(G)\setminus(I\cup K\cup S_{j}))$. To achieve this, we partition the vertices in $V(G)\setminus (I\cup K\cup S_{j})$ into two groups based on $|N_{I}(v)|$. More precisely, let
\begin{equation*}
    A_{L}:=\{v\in V(G)\setminus(I\cup K\cup S_{j}): |N_{I}(v)|\ge n(\log{n})^{-(10s)^{2j-1}}\},
\end{equation*}
and
\begin{equation*}
    A_{S}:=\{v\in V(G)\setminus(I\cup K\cup S_{j}): |N_{I}(v)|< n(\log{n})^{-(10s)^{2j+1}}\}.
\end{equation*}
Since $|V(G)\setminus(I\cup K\cup S_{j})|\le (1-c) n$, the expected number of edges in $E(I,A_{S})$ is less than
\begin{equation*}
    n(\log{n})^{-(10s)^{2j+1}}\cdot (1-c) n  <  \frac{1}{10}\cdot\bigg(\frac{\delta cn^{2}}{2(t-1)\binom{k_{j}}{s}+2}-cn\bigg).
\end{equation*}
It remains to estimate the expected number of edges in $E(I,A_{L})$. The key observation is that, for a vertex $v\in V(G)\setminus I$ such that $|N_{I}(v)|\ge n(\log{n})^{-(10s)^{2j-1}}$, the probability that $v\in A_{L}$ is equal to the probability that $I_{j}$ contains at most $s-1$ vertices in $N_{I}(v)$. Therefore, the expected number of edges in $E(I,A_{L})$ is at most 
\begin{equation*}
\begin{split}
& \sum\limits_{\substack{v\in V(G)\setminus I:\\|N_{I}(v)|\ge n(\log{n})^{-(10s)^{2j-1}}}}\sum\limits_{x=0}^{s-1}\mathbb{P}\bigg[|I_{j}\cap N_{I}(v)|=x\bigg]\cdot |N_{I}(v)| \\
&\le \sum\limits_{\substack{v\in V(G)\setminus I:\\|N_{I}(v)|\ge n(\log{n})^{-(10s)^{2j-1}}}}\sum\limits_{x=0}^{s-1}\binom{k_{j}}{x}\bigg(1-\frac{|N_{I}(v)|}{cn}\bigg)^{k_{j}-x}\bigg(\frac{|N_{I}(v)|}{cn}\bigg)^{x}\cdot cn \\
&\le c(1-c)n^{2}\cdot 
\sum\limits_{x=0}^{s-1}\binom{k_{j}}{x}\bigg(1-\frac{n(\log{n})^{-(10s)^{2j-1}}}{cn}\bigg)^{k_{j}-x}\bigg(\frac{cn}{cn}\bigg)^{x}\\
&\le c(1-c)n^{2}\cdot 
\sum\limits_{x=0}^{s-1}k_{j}^{x}\bigg(1-\frac{(\log{n})^{-(10s)^{2j-1}}}{c}\bigg)^{k_{j}-x}\\
&\le \frac{1}{10} \cdot\bigg(\frac{\delta cn^{2}}{2(t-1)\binom{k_{j}}{s}+2}-cn\bigg).
\end{split}
\end{equation*}
By the linearity of expectation, the expected number of $e$ is smaller than $\frac{\delta cn^{2}}{2(t-1)\binom{k_{j}}{s}+2}-cn$, which implies the existence of a choice of $I_{j}$ such that $|H\cup N(H)\cup S_{j}|<\delta n$. The proof is then completed by~\cref{claim:Hitting}.

\begin{comment}
The only thing remain is to compute the expectation of $ek_{j}^2$, if it is far less than $\delta n^2$, then we are done. Let $e_{large}$ to be the edge set remains in the high degree vertices (degree in $I$ larger than $n(\log n)^{-10^ {2j-1}}$), and $e_{small}$ to be the edge set remains in the low degree vertices (degree in $I$ less than $n(\log n)^{-10^ {2j+1}}$). Since we remove all vertices in $V(G)\setminus I$ that have middle degree, we only need to consider $e_{small}$ and $e_{large}$. So we get

$$E(ek_{j}^2)=E((e_{large}+e_{small})k_{j}^2)$$

For $E(e_{small}k_{j}^2)$, since the total number of vertices is less than $n$, $e_{small}k_{j}^2 < n^2(\log n)^{-10^ {2j+1}}(\log n)^{2 \times 10^ {2j}} $ always far less than $\delta n^2$.

For $E(e_{large}k_{j}^2)$, we pick one edge $e_{0}=ab$ in high degree vertices, if there is only one or zero vertices in $V_{j}$ comes to $N_{I}(b)$, $e_{0}$ will survive, so the probability of a specific edge $e_{0}$ survives is about $k_{j}(1-N_{I}(b)/cn)^{k_{j}-1}$, since the number of edges is less than $n^2$, we get 
$$E(e_{large}k_{j}^2) < n^2(\log n)^{3\times 10^ {2j}}(1-(1/c)(\log n)^{-10^ {2j-1}})^{(\log n)^{10^ {2j}}-1}.$$
far less than $\delta n^2$.
\end{comment}

\section{Some remarks}
\subsection{More on the proof}
We originally thought that the following result would be helpful for our proof, as it was proved by Gy\'{a}rf\'{a}s, Hubenko, and Solymosi~\cite{2002Solymosi} and more recently by Holmsen~\cite{2020Holmsen} via a different approach. We also recommend interested readers to refer to~\cite{2017Mindegree,2021ECJNoK2t,2018CPCInducedTuran}. Let $\omega(G)$ denote the size of the largest complete subgraph in a graph, generally referred to as the \emph{clique number} of the graph $G$.
 
\begin{theorem}[\cite{2002Solymosi,2020Holmsen}]\label{thm:CliqueC4Free}
Let $G$ be an $n$-vertex induced $C_{4}$-free graph with $e(G)\ge\alpha\binom{n}{2}$ for some $\alpha>0$, then there is some $\beta>0$ such that $\omega(G)\ge \beta n$.    
\end{theorem}
We aimed to systematically remove cliques of linear size from an induced $C_{4}$-free graph, with the goal of eventually reaching a contradiction. However, the quantitative result in~\cref{thm:CliqueC4Free} indicates that repeatedly finding cliques of appropriate sizes is not feasible. Therefore, in the actual proof of~\cref{thm:Kst}, we employ the probabilistic method. Nevertheless, we believe that such a result is relevant to our main topic, so we present it here along with a self-contained proof using dependent random choice~\cite{2011DRC}.

\begin{proof}[Proof of~\cref{thm:CliqueC4Free}]
Let $\beta:=\frac{\alpha^{2}}{128}$.
We call $uv$ a \emph{missing edge} if $uv\notin E(G)$. As $G$ is induced $C_{4}$-free, for any missing edge $uv$, $G[N(u)\cap N(v)]$ must form a clique. We can assume that for every missing edge $uv$, $|N(u)\cap N(v)|<\beta\cdot n$, otherwise we are done. Pick one vertex $x\in V(G)$ uniformly at random, let $U:=N(x)$. It is easy to see that $\mathbb{E}[|U|]=\sum\limits_{v\in V(G)}\frac{|N(v)|}{n}\ge\alpha (n-1)$.
Let $Y$ be the random variable counting the number of missing edges in $U$, then we can see the expected number $\mathbb{E}[Y]<\frac{\beta n}{n}(1-\alpha)\binom{n}{2}=\beta(1-\alpha)\binom{n}{2}$, where the first inequality is due to the probability that a missing edge belongs to $U$ is at most $\frac{\beta n}{n}$. Let $Z:=|U|-\frac{\alpha\cdot Y}{\beta (1-\alpha)n}-\frac{\alpha(n-1)}{2}$. By linearity of expectation, $\mathbb{E}[Z]\ge 0$ and therefore there is a choice of $x$ such that $|U|\ge\frac{\alpha(n-1)}{2}$ and $\frac{Y}{|U|}\le\frac{\beta(1-\alpha)n}{\alpha}$. 

Fix this vertex $x$, and take a maximal matching of missing edges in $U=N(x)$, namely, $\tau_{1},\tau_{2},\ldots,\tau_{s}$. Since $G$ is induced $C_{4}$-free, for every pair $\tau_{i}$ and $\tau_{j}$, there is at least one more missing edge sharing one vertex with $\tau_{i}$ and $\tau_{j}$, respectively, which implies that the number of missing edges in $U$ is at least $s+\binom{s}{2}$. One the other hand, we have $s+\binom{s}{2}\le |Y|\le \frac{\beta(1-\alpha)n}{\alpha}\cdot |U|$.
Also note that $\omega(G)\ge |U|-2s+1$, then the result follows by simple calculations.
\end{proof}

\begin{rmk}
    Determining the optimal constant dependence between $\alpha$ and $\beta$ in~\cref{thm:CliqueC4Free} is of interest. The current best-known result is $\beta=(1-\sqrt{1-\alpha})^{2}$, which was shown in~\cite{2020Holmsen}
\end{rmk}

\subsection{Graphs without other induced substructures}
In general, we can show that~\cref{conj:BETConj} holds for those graphs having linear independence number but the number of maximum independent sets is relatively small.
\begin{theorem}\label{thm:SparseInd}
    Let $G$ be an $n$-vertex graph with $\alpha(G)=\Omega(n)$. If the number of distinct maximum independent sets is $2^{o(n)}$, then $h(G)=o(n)$.
\end{theorem}
There are two simple approaches to showing~\cref{thm:SparseInd}. The first is the same as the proof in~\cite[Theorem~1.5]{2021AlonHittingSet}, one can randomly sample $p:=o(n)$ vertices and the probability that a maximum independent set is not hit is at most $(1-\frac{p}{n})^{cn}$. Using a union bound, then one can show with a positive probability, that the randomly chosen vertices hit all maximum independent sets. The second proof is similar to the proof in~\cite[Theorem~1.2]{2024HitInducedMathcing}. One can build a set system $\mathcal{F}$ consisting of all maximum independent sets in $G$. It then suffices to upper bound the transversal number of $\mathcal{F}$, which in turn is upper bounded by a function of its fractional transversal number and VC-dimension, due to the well-known $\varepsilon$-net theorem in~\cite{1987DCG}. Furthermore, the fractional transversal number is a constant due to $\alpha(G)=\Omega(n)$, and the VC-dimension is $o(n)$ since the number of maximum independent sets is $2^{o(n)}$. We omit the details and straightforward calculations of both proofs here.

In particular, It was shown in~\cite{1989M2free,farber1993upper} that the number of maximum independent sets in an $n$-vertex graph that does not contain an induced matching of size $t$ is at most $O(n^{2t-2})$. Therefore, for an $n$-vertex graph $G$ with no induced matching of size $t$, if $\alpha(G)=\Omega(n)$, then $h(G)=O(\log{n})$.

We are also interested in whether~\cref{conj:BETConj} holds for induced-$H$-free graphs when $H$ is not a complete bipartite graph. For instance, it will be interesting to generalize the result on $4$-cycle to longer even cycles.
\begin{conj}
    For a positive integer $k\ge 3$, let $G$ be an $n$-vertex induced $C_{2k}$-free graph with $\alpha(G)=\Omega(n)$, then $h(G)=o(n)$.
\end{conj}
A result of Brandt~\cite{2002DAMK3C6} states that if $G$ is a maximal triangle-free graph without induced $6$-cycle, then every independent set of $G$ is contained in the neighborhood of a vertex. Therefore, the number of maximum independent sets in such a graph is at most $n$, then one can take $O(\log{n})$ vertices to hit all maximum independent sets.

\bibliographystyle{abbrv}
\bibliography{BETC4}
\end{document}